\documentclass[12pt,a4paper]{article}

\usepackage[dvips]{graphicx}
\usepackage{amssymb,latexsym}
\usepackage{mathrsfs}
\usepackage{psfrag}
\usepackage[all]{xy}
\usepackage{color}
\usepackage{verbatim}
\usepackage{graphicx,geometry,color,epsfig,epstopdf}
\usepackage{caption}
\usepackage{subcaption}

\newtheorem{prop}{Proposition}[section]
\newtheorem{theo}[prop]{Theorem}
\newtheorem{cor}[prop]{Corollary}

\newtheorem{rem}[prop]{Remark}

\newtheorem{lem}[prop]{Lemma}

 \newcommand{\bx}{\mbox{\boldmath $x$}}

\newcommand{\bv}{\mbox{\boldmath $v$}}

\newcommand{\bu}{\mbox{\boldmath $u$}}

\newcommand{\bN}{\mbox{\boldmath $N$}}

\newcommand{\R}{\mbox{\boldmath $\mathbb{R}$}}

\newenvironment{proof}
{\begin{trivlist} \item[\hskip \labelsep {\bf Proof}\hspace*{3 mm}]}
	{\hfill$\Box$\end{trivlist}}
\newenvironment{acknow}
{\begin{trivlist} \item[\hskip \labelsep {\bf Acknowledgments.}]}
	{\end{trivlist}}

\date{}

\begin{document}

\title{Möbius inversion of surfaces in the Minkowski 3-space}
\author{Marco Ant\^onio do Couto Fernandes}

\maketitle
\begin{abstract}
We define and present some proprieties of the Möbius inversion of surfaces in the Minkowski 3-space. We prove that the Möbius inversion preserves the lines of principal curvature and the locus of points where the metric is degenerate, but it does not preserve the parabolic set. For ovaloids, we show that it is possible to translate the surface so that the inversion remains an ovaloid.

\end{abstract}

\renewcommand{\thefootnote}{\fnsymbol{footnote}}
\footnote[0]{2020 Mathematics Subject classification:
	53A35 
	53A05 
	53C50 
	52A15 
}
\footnote[0]{Key Words and Phrases. Möbius Inversion, Minkowski space, Surfaces, Curvature Lines.}

\section{Introduction}\label{sec:intro}

The Möbius inversion of a closed subset $F$ in the Euclidean space $\R^3$ is given by
$$i_E(F)= cl \left \{ \frac{p}{\langle p,p \rangle_E}  : p \in F \setminus \{0\} \right \},$$
where $cl$ denotes the topological closure in $\R^3$ endowed with the topology induced by the Euclidean inner product $\langle \cdot,\cdot\rangle_E$. Geometrically, points on the unit Euclidean sphere centered at the origin are fixed points of $i_E$. Moreover, the image of points inside this sphere are points outside it and vice-versa.

The Möbius inversion preserves the lines of principal curvature on surfaces in $\R^3$. Ghomi and Howard \cite{GhomiHoward} use this fact to relate the lines of principal curvature of a closed and convex surface with one umbilic point removed to the lines of principal curvature of the graph of an asymptotically constant function. That lead to give the following new formulation of the Carathéodory Conjecture: the graph of an asymptotically constant function $f:\R^2 \to \R$ has at least one umbilic point.

In this paper, we define the Möbius inversion of surfaces in the Minkowski 3-space $\R^3_1$ and show that it preserves their lines of principal curvature. Furthermore, the Möbius inversion preserves the set of points where the metric is degenerate and the discriminant of the lines of principal curvature.

We show that the parabolic set is not preserved by the Möbius inversion by means of a counterexample. However, when we work with ovaloids, which are closed surfaces with empty parabolic sets, we obtain conditions for the Möbius inversion of these surfaces to remain ovaloids.

The paper is organized as follows. In Section \ref{sec:prel}, we present the Minkowski 3-space $\R^3_1$ and some concepts of the differential geometry of surfaces in $\R^3_1$. In Section \ref{sec:mobius_R31}, we define the Möbius inversion in $\R^3_1$ and give some of its properties. Finally, in Section \ref{sec:parabolic}, we show that the parabolic set is not always preserved by Möbius inversions, but for ovaloids it is possible to impose conditions in order to preserve the emptiness of the parabolic set.


\section{Preliminaries} \label{sec:prel}

The {\it Minkowski 3-space} $(\mathbb{R}_1^{3},\langle ,\rangle )$ is the vector space $\mathbb{R}^{3}$ endowed with the Minkowski metric induced by the pseudo-scalar product $ \langle \bu, \bv\rangle =u_0v_0+u_1v_1-u_2v_2$, for any vectors $\bu=(u_0,u_1,u_2)$ and $\bv =(v_0,v_1, v_2)$ in $\mathbb{R}^{3}$. A non-zero vector $\bu\in \mathbb R^3_1$ is said to be {\it spacelike} if $\langle \bu, \bu\rangle>0$, {\it lightlike} if $\langle \bu, \bu\rangle=0$ and {\it timelike} if $\langle \bu, \bu\rangle<0$. The norm of a vector $\bu\in \mathbb{R}_1^{3}$ is defined by
$\Vert \bu \Vert=\sqrt{ |\langle \bu, \bu\rangle|}.$

If $p$ is a point in $\R^3$, then $\langle p,p \rangle$ denotes $\langle \vec{0p},\vec{0p} \rangle$, where $\vec{0p}$ is the vector from the origin to $p$. The set of points $p \in \R^3_1$ with $\langle p,p \rangle = 0$ is called the \textit{light cone} and is denoted by $LC$. The \textit{de Sitter sphere}, denoted by $\mathbb{S}^2_1$, is the set of points $ q \in \R^3_1$ such that $\langle q, q \rangle = 1$. On the other hand, the \textit{hyperbolic plane}, denoted by $\mathbb{H}^2$, is the set of the points $q \in \R^3_1$ where $\langle q, q \rangle = -1$.

Let $S$ be a smooth and regular surface in $\mathbb R^3_1$ and let 
$\bx:U\subset \mathbb R^2\to \mathbb{R}_1^{3}$ be a local parametrization of $S$.
The concepts of differential geometry in $\R^3$ can be extended to $\R^3_1$, see for example \cite{Couto_Lymb}. The pseudo inner product $\langle \cdot,\cdot \rangle$ in $\R^3_1$ induces a pseudo metric on $S$ called the first fundamental form. Let $E=\langle {\bx}_u,{\bx}_u \rangle$, $F=\langle{\bx}_u,{\bx}_v\rangle$ and $G=\langle {\bx}_v,{\bx}_v\rangle$, where subscripts denote partial derivatives. The induced (pseudo) metric on $S$ is Lorentzian (resp. Riemannian, degenerate) at $p=\bx(u,v)$ if and only if $\delta(u,v) = (F^2-EG)(u,v)>0$ (resp. $<0$, $=0$). The locus of points on the surface where the metric is degenerate is called the {\it locus of degeneracy} and is denoted by $LD$. We identify the $LD$ on $S$ with its pre-image in $U$ by $\bx$.

At $p \in S\setminus LD$, we have a well defined unit normal vector (the Gauss map) $\bN=\bx_u\times\bx_v/||\bx_u\times\bx_v||$,
which is timelike (resp. spacelike) if $p$ is in the Riemannian (resp. Lorentzian) region of $S$. The self-adjoint operator $A_p=-d\bN_p:T_pS\to T_pS$ on $S\setminus LD$ defines a quadratic form $II_p(\bv) = \langle A_p(\bv),\bv \rangle$ on $T_pS$ called the second fundamental form. We denote by $l=\langle \bN,\bx_{uu}\rangle$, $m=\langle \bN,\bx_{uv}\rangle$ and $n=\langle \bN,\bx_{vv}\rangle$ the coefficients of the second fundamental form on $S \setminus LD$.

When $A_p$ has real eigenvalues $\kappa_1$ and $\kappa_2$, we call them the {\it principal curvatures} and their associated eigenvectors the {\it principal directions}
of $S$ at $p$. There are always two principal curvatures at each point on the Riemannian part of $S$, but this is not always the case on its Lorentzian part.
A point $p$ on $S$ is called an {\it umbilic point} if  $\kappa_1=\kappa_2$ at $p$ (i.e., if $A_p$ is a multiple of the identity map).

The \textit{lines of principal curvature} on $S$, which are the integral curves of the principal directions, are the solutions of the binary differential equation (BDE)
\begin{equation}\label{eq:principalBDE}
	(Fn-Gm)dv^2+(En-Gl)dvdu+(Em-Fl)du^2=0.
\end{equation}
The discriminant of the BDE (\ref{eq:principalBDE}) consists of the umbilic points in the Riemannian region of $S$ and is the locus of points in the Lorentzian region where two principal directions coincide and become lightlike. In the last case, it is labelled {\it Lightlike Principal Locus} $(LPL)$ in \cite{Izu-Machan-Farid}. 

The \textit{Gaussian curvature} of $S$ at $p$ is given by $K(p) = \det (A_p)$, that is,
$$K(p) = \frac{ln-m^2}{EG-F^2}(p).$$
The locus of points where the Gaussian curvature is zero is called the \textit{parabolic set}.

One can extend the lines of principal curvature, the $LPL$ and the parabolic set across the $LD$ as follows (\cite{IzumiyaTari}).
As equation (\ref{eq:principalBDE}) is homogeneous in $l,m,n$,
we can multiply these coefficients by $||\bx_u\times\bx_v||$ and substitute them in equation (\ref{eq:principalBDE}) by
\[
\bar{l}=\langle \bx_u\times\bx_v,\bx_{uu}\rangle, \quad
\bar{m}=\langle \bx_u\times\bx_v,\bx_{uv}\rangle,\quad
\bar{n}=\langle \bx_u\times\bx_v,\bx_{vv}\rangle.
\]
The new equation 
\begin{equation}\label{eq:principalLD}
	(G\bar{m}-F\bar{n})dv^2+(G\bar{l}-E\bar{n})dudv+(F\bar l-E\bar{m})du^2=0
\end{equation}
is defined at points on the $LD$ and
its solutions are the same as those of equation (\ref{eq:principalBDE}) in the Riemannian and Lorentzian regions of $S$. The discriminant of the BDE $(\ref{eq:principalLD})$ extends the $LPL$ to points on the $LD$. Thus, a point $p = \bx(u,v)$ belongs to the $LPL$ of $S$ when
$$ \tilde \delta(u,v) = \left((E\bar n-G\bar l)^2-4(F\bar n-G\bar m)(E\bar m-F\bar l)\right)(u,v)=0. $$
For the parabolic set, note that $K = 0$ if and only if $\bar l \bar n-\bar{m}^2=0$. Thus, the parabolic set also extends to points on the $LD$ as  the zero set of the function $\bar K=\bar l \bar n-\bar{m}^2$.

\medskip

\section{Möbius inversion in $\R^3_1$}\label{sec:mobius_R31}

We define the \textit{Möbius inversion} of a closed subset $F \subset \R^3_1$ as
$$i_M(F) = cl \left \{ \frac{p}{\langle p , p \rangle} : p \in F \setminus LC \right \},$$
where $cl$ denotes the topological closure in $\R^3$ with the same topology used to define $i_E$. When necessary, we consider the Möbius inversion as a function $i_M: \R^3_1 \setminus LC \to \R^3_1 \setminus LC$ given by $i_M(p) = p/\langle p,p \rangle$. Consider the following regions of $\R^3_1$:
\begin{equation}\label{eq:Ri}
\begin{array}{l}
R_1 = \{ (x,y,z) \in \R^3_1 : x^2+y^2-z^2>0 \},\\ R_2 = \{ (x,y,z) \in \R^3_1 : x^2+y^2-z^2<0, \, z>0\},\\
R_3 = \{ (x,y,z) \in \R^3_1 : x^2+y^2-z^2<0, \, z<0\}.
\end{array}
\end{equation}
Then $i_M$ is a diffeomorphism from $R_1$ to $R_1$ and from $R_2$ to $R_3$ with $i_M^{-1}=i_M$. Points on the hyperbolic plane $\mathbb{H}^2$ and on the de Sitter sphere $\mathbb{S}^2_1$ are fixed points of $i_M$. In $R_1$, $i_M$ is an inversion with respect to $\mathbb{H}^2$, Figure \ref{fig:mobius_inv_R31} left. On the other hand, in $R_2$ and $R_3$, $i_M$ is an inversion with respect to the $\mathbb{S}^2_1$, Figure \ref{fig:mobius_inv_R31} center and right.
\begin{figure} [h!]
	\centering
	\includegraphics[scale=0.4]{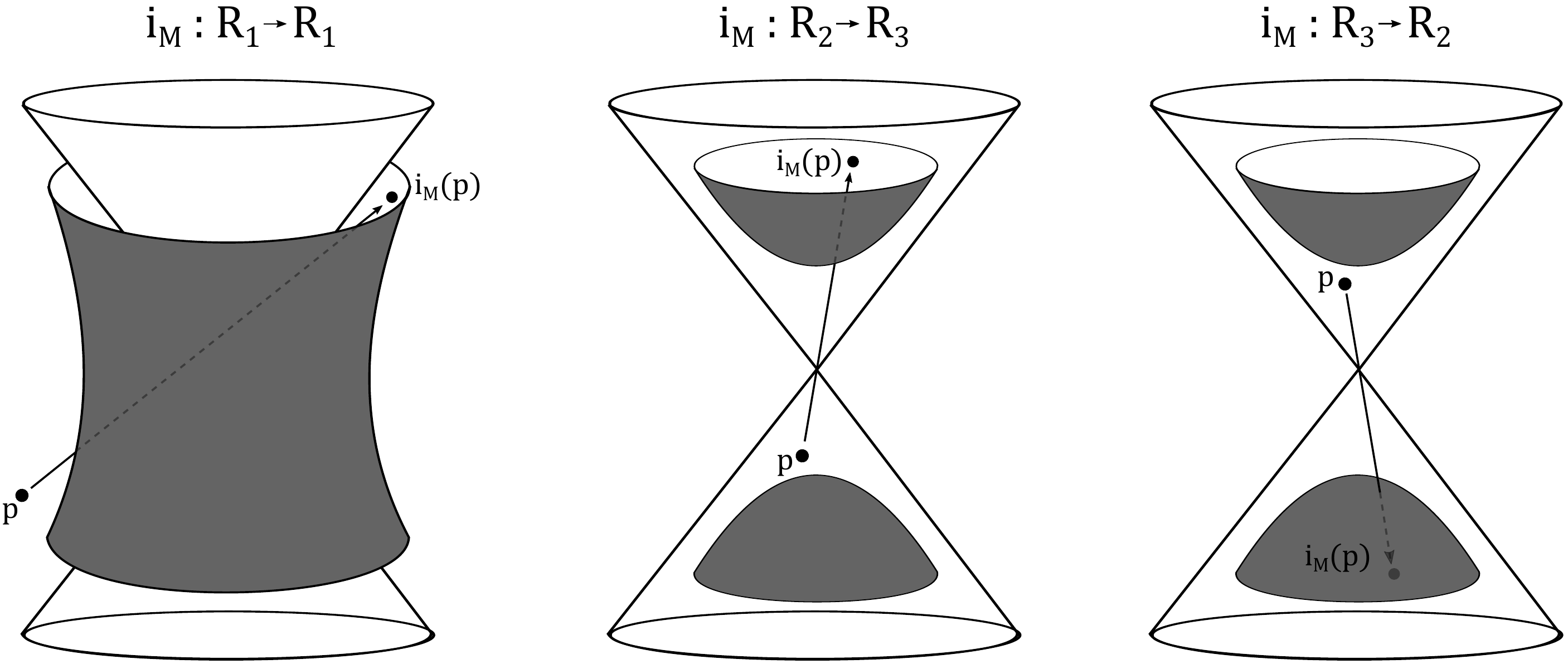}
	\caption{Möbius inversion in Minkowski space.}
	\label{fig:mobius_inv_R31}
\end{figure}

\begin{theo}\label{prop:inv_compac}
	Let $S$ be a compact surface in $\R^3_1$. Then following properties hold:
	\begin{enumerate}
		\item $i_M(S) = \left \{ \frac{p}{\langle p, p \rangle}  : p \in S \setminus LC \right \}$;
		\item $i_M(S) \cap LC = \emptyset$;
		\item $i_M(S)$ is bounded if and only if $S \cap LC = \emptyset$.
	\end{enumerate}
\end{theo}

\begin{proof}
	(1) Let $A = \left \{ \frac{p}{\langle p, p \rangle} : p \in S \setminus LC \right \}$ and $\{q_i\}_{i \in \mathbb{N}}$ a sequence in $A$ that converges to some $q \in \R^3_1$, with $q_i = p_i/\langle p_i,p_i \rangle$ and $p_i \in S \setminus LC$, for all $i \in \mathbb{N}$.
	Since $S$ is compact, there is a subsequence $\{p_{i_k}\}$ of $\{p_i\}$ that converges to some $p \in S$. If $p \in LC$, then $\| p_{i_k} \| \rightarrow 0$ and $\| q_{i_k} \| \rightarrow \infty$, which contradicts the convergence of $\{q_i\}$. Thus, $p \in S \setminus LC$ and
	$$q = \lim_{k \rightarrow \infty} q_{i_k} = \lim_{k \rightarrow \infty} \frac{p_{i_k}}{\langle p_{i_k},p_{i_k} \rangle} = \frac{p}{\langle p,p \rangle} \in A.$$
	Therefore, $A$ is closed and $i_M(S) = cl(A) = A$.\vspace{0.2cm}
		
	\noindent (2) Suppose that $q \in i_M(S) \cap LC$. It follows from (1) above that there exists $p \in S \setminus LC$ such that $q = i_M(p)$. Then
	$$0 = \langle q,q \rangle = \left \langle \frac{p}{\langle p,p \rangle},\frac{p}{\langle p,p \rangle} \right \rangle = \frac{1}{\langle p,p \rangle},$$
	which is not possible since $\langle p,p \rangle < \infty$. Therefore, $i_M(S)$ does not intersect $LC$.\vspace{0.2cm}
		
	\noindent (3) If $S \cap LC \neq \emptyset$, then there is a sequence $\{ p_i \}_{i \in \mathbb{N}}$ in $S \setminus LC$ and $p \in S \cap LC $ with $p_i \to p$. Then
	$$\lim_{i \rightarrow \infty} \| p_i \| = \| p \| = 0 \Rightarrow \lim_{i \rightarrow \infty} \| i_M(p_i) \| = \infty,$$
	so $i_M(S)$ is unbounded. For the converse, if $S \cap LC = \emptyset$, then
	$$\alpha = \sup \left\{ \frac{1}{|\langle p,p \rangle|}, p \in S \right\} < \infty \quad {\rm and} \quad \beta = \sup \{ \| p \| , p \in S \} < \infty.$$
	Therefore, $i_M(S)$ is contained in the Euclidean sphere with center the origin and radius $\alpha \beta$.
\end{proof}


An important property of the Möbius inversion in the Euclidean space $\R^3$ is that it preserves the lines of principal curvature of surfaces in $\R^3$, that is, if $C$ is a line of principal curvature of a surface $S \subset \R^3$, then $i_M(C)$ is a line of principal curvature of the surface $i_M(S)$. That property also hold for Möbius inversion in $\R^3_1$.

\begin{theo}\label{teo:dp}
	The lines of principal curvature of surfaces in $\R^3_1$ are preserved by the Möbius inversion.
\end{theo}

\begin{proof}
	Let $S \subset \R^3_1$ be a surface and $\varphi : U \subset \R^2 \rightarrow \R^3_1$ be a local parameterization of $S$. Then $\tilde \varphi = i_M \circ \varphi$ is a local parameterization of $S_M = i_M(S)$, where $E$, $F$ and $G$ are the coefficients of the first fundamental form of $S$ and $E_M$, $F_M$ and $G_M$ are those of $S_M$. It follows that
	\begin{equation}\label{eq:pff}
	E_M = \frac{E}{\langle \varphi, \varphi \rangle^2}, \quad F_M = \frac{F}{\langle \varphi, \varphi \rangle^2}, \quad G_M = \frac{G}{\langle \varphi, \varphi \rangle^2}.
	\end{equation}
	The coefficients $\bar{l}_M$, $\bar{m}_M$ and $\bar{n}_M$ of the second fundamental form of $S_M$ with denominator removed are expressed as follows in terms of the coefficients $\bar{l }$, $\bar{m}$ and $\bar{n}$ of $S$:
	$$\bar{l}_M = \frac{1}{\langle \varphi, \varphi \rangle^3} (\bar{l}+\alpha E), \qquad \bar{m}_M = \frac{1}{\langle \varphi, \varphi \rangle^3} (\bar{m}+\alpha F), \qquad \bar{n}_M = \frac{1}{\langle \varphi, \varphi \rangle^3} (\bar{n}+\alpha G),$$
	with $\alpha = -2 \langle N, \tilde \varphi \rangle$.
	
	The equation of the lines of principal curvature of $S_M$ is given by
	\begin{equation}\label{eq:linha_curv_mobius}
	(\bar{m}_M E_M - \bar{l}_M F_M) dx^2+(\bar{n}_M E_M-\bar{l}_M G_M)dxdy+(\bar{n}_M F_M-\bar{m}_M G_M) dy^2 = 0.    
	\end{equation}
	Substituting the coefficients $E_M$, $F_M$, $G_M$, $\bar l_M$, $\bar m_M$ and $\bar n_M$ by their expressions in terms of $E$, $F$, $G$, $\bar l$, $\bar m$ and $\bar n$, it follows that (\ref{eq:linha_curv_mobius}) is satisfied if and only if
	$$(\bar f E - \bar e F) dx^2+(\bar g E-\bar e G)dxdy+(\bar g F- \bar f G) dy^2 = 0,$$
	which is precisely the equation of the lines of principal curvatures of $S$.
\end{proof}

\begin{cor}\label{cor:dp}
	The LPL and the LD of a surface in the Minkowski 3-space are preserved by the Möbius inversion.
\end{cor}

\begin{proof}
	The proof for the preservation of the LPL is consequence of the Theorem \ref{teo:dp}. As for the LD, observe from (\ref{eq:pff}) that $F_M^2-E_M G_M = 0$ if and only if $F^2-EG=0$.
\end{proof}

\section{Möbius inversion and the parabolic set}\label{sec:parabolic}

Let $p_0 = (a,b,c) \in \R^3_1$ and 
$$S(p_0,r) = \{ p \in \R^3_1 : \langle p-p_0,p-p_0 \rangle_E = r^2 \},$$
so $S(p_0,r)$ is the Euclidean sphere with center $p_0$ and radius $r > 0$ considered as a set of points in $\R^3_1$. Denote $S_M(p_0,r) = i_M(S(p_0,r))$.

\begin{theo}\label{teo:closed_surface}
	The surface $S_M(p_0,r) = i_M(S(p_0,r))$ is homeomorfic to $S(p_0,r)$ if and only if $(\sqrt{a^2+b^2} - \vert c \vert)^2>2r^2$.
\end{theo}

\begin{proof}
	From item (3) of Theorem \ref{prop:inv_compac} and from the fact that $i_M : \R^3_1 \setminus LC \to \R^3_1 \setminus LC$ is a diffeomorphism, it follows that $S_M(p_0,r)$ is homeomorfic to $S(p_0,r)$ if and only if $S(p_0,r) \cap LC = \emptyset$. Note that $S(p_0,r) \cap LC = \emptyset$ when $d(p_0,LC) > r$, where $d(p_0,LC)$ is the Euclidean distance between $p_0$ and $LC$. 
	
	We need to obtain the minimum value of the function $$f(x,y,z) = (x-a)^2+(y-b)^2+(z-c)^2$$ restricted to the light cone $x^2+y^2=z^2$. Using the Lagrange Multipliers method, we find the minimum point and it follows that
	$$d(p_0,LC) = \frac{\sqrt{(\sqrt{a^2+b^2} - \vert c \vert)^2}}{\sqrt{2}}.$$
	Therefore, $S(p_0,r) \cap LC = \emptyset$ if and only if $(\sqrt{a^2+b^2} - \vert c \vert)^2 > 2r^2$.
\end{proof}

The surface $S_M(p_0,r)$ with $p_0=(2,0,0)$ and $r=1$ is closed (ie, compact without boundary) by Theorem \ref{teo:closed_surface} because the point $\left( \frac{3}{4},\frac{\sqrt{15}}{12},0 \right) \in S_M(p_0,r)$, for example, belongs to the parabolic set, see Figure \ref{fig:inv_esfera}. Thus, unlike the lines of principal curvature, the $LD$ and the $LPL$, the parabolic set is not preserved by the Möbius inversion. We have the following conditions on $p_0$ and $r$ for $S_M(p_0,r)$ to be convex.
\begin{figure} [h!]
	\centering
	\includegraphics[scale=0.45]{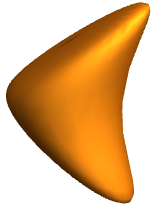}
	\caption{The surface $S_M(p_0,r)$ with $p_0=(2,0,0)$ and $r=1$.}
	\label{fig:inv_esfera}
\end{figure}

\begin{theo}\label{teo:inv_cp}
	The parabolic set of $S_M(p_0,r)$ is empty if and only if $\sqrt{a^2+b^2} > 2 r + \sqrt{c^2 + r^2}$ or $|c| > 2 r + \sqrt{a^2 + b^2 + r^2}$.
\end{theo}

\begin{proof}
	Let $L = S(p_0,r) \cap LC$. Take a local parameterization $\varphi : U \to \R^3_1$ of $S(p_0,r)$ given by $\varphi (u,v) = \left( r\cos (u) \sin (v) + a, r\sin (u) \sin(v) + b, r\cos(v) +c \right)$, with $U = (\epsilon_0,\epsilon_0+2 \pi) \times (0,\pi)$ for some $\epsilon_0 > 0$ small enough. Then $\tilde \varphi : U \setminus L \to S_M(p_0,r)$ is a local parameterization of $S_M(p_0,r)$ with $\tilde \varphi(u,v) = i_M(\varphi(u,v))$. Calculating the Gaussian curvature of $S_M(p_0,r)$, the point $\tilde \varphi(u,v)$ belongs to the parabolic set of $S_M(p_0,r)$ if and only if $f(u,v) = 0$ or $g(u,v) = 0$, with
	$$\begin{array}{ccl}
	f(u,v) & = & - a^2 - b^2 + c^2 + r^2 +4cr \cos (v) + 2 r^2 \cos(v)^2, \vspace{0.2cm}\\
	g(u,v) & = & -a^2-b^2+c^2+4cr \cos(v)^3 + 3 r^2 \cos (2v) - 4r h(u) \sin(v)^3,
	\end{array}$$
	and $h(u) = a \cos (u)+b \sin (u)$.
	
	Note that $f(u,v) = 0$ if and only if
	\begin{equation}\label{eq:cos_v}
	\cos(v) = \frac{-2 c\pm\sqrt{2} \sqrt{a^2+b^2+c^2-r^2}}{2 r}.
	\end{equation}
	Equation (\ref{eq:cos_v}) has no solutions when $a^2+b^2+c^2<r^2$, $|c| < -2r+\sqrt{a^2+b^2+r^2}$ or $|c| > 2 r+\sqrt{a^2+b^2+r^2}$.
	
	
	The image of the function $h$ is the closed interval $\left[ -\sqrt{a^2+b^2},\sqrt{a^2+b^2} \right]$, so $g_{min}(v) \leq g(u,v) \leq g_{max}(v)$ for all $(u,v)$ in $U$, with
	$$\begin{array}{ccl}
		g_{min}(v) = -a^2-b^2+c^2+4cr \cos(v)^3 + 3 r^2 \cos (2v) - 4r \sqrt{a^2+b^2} \sin(v)^3, \vspace{0.2cm}\\
		g_{max}(v) = -a^2-b^2+c^2+4cr \cos(v)^3 + 3 r^2 \cos (2v) + 4r \sqrt{a^2+b^2} \sin(v)^3.
	\end{array}$$
	Computing the critical values of $g_{min}$ in $[0,\pi]$, we get that $g_{min} > 0$ when $|c| > 2 r + \sqrt{a^2+b^2+r^2}$. Likewise, $g_{max} < 0$ when $\sqrt{a^2+b^2} > 2 r + \sqrt{c^2 + r^2}$. Hence, $g(u,v) = 0$ has no solutions if and only if $\sqrt{a^2+b^2} > 2 r + \sqrt{c^2 + r^2}$ or $|c| > 2 r + \sqrt{a^2 + b^2 + r^2}$.
	
	The parameterization $\varphi$ does not include the points $(0,0,\pm r) \in S(p_0,r)$. However, using other parameterization of $S(p_0,r)$, it follows that $i_M(0,0,\pm r)$ does not belong to the parabolic set of $S_M(p_0,r)$ when $\sqrt{a^2+b^2} > 2 r + \sqrt{c^2 + r ^2}$ or \linebreak $|c| > 2 r + \sqrt{a^2 + b^2 + r^2}$.
\end{proof}

A surface $S$ has a contact $A_1^+$ with $T_pS$ at $p \in S$ when the height function $h$ has a Morse singularity of index 0 or 2 at $p$. A closed surface $S$ is an \textit{ovaloid} if it is strictly convex, that is,it has a contact $A_1^+$ with the tangent plane $T_pS$, for all $p \in S$. Note that the definition of ovaloid is independent of the metric (Euclidean or Lorentzian) in $\R^3$. For example, the Euclidean spheres are ovaloids.

\begin{rem}\label{teo:ovaloide}
	A closed surface is an ovaloid if and only if its parabolic set is empty. Indeed, an ovaloid consists of eliptic points.
\end{rem}

%
%

\begin{cor}\label{cor:mSac_ov}
	$S_M(p_0,r)$ is an ovaloid if and only if $\sqrt{a^2+b^2} > 2 r + \sqrt{c^2 + r^2}$ or $|c| > 2 r + \sqrt{a^2 + b^2 + r^2}$.
\end{cor}

\begin{proof}
	The proof follows from the Theorems \ref{teo:closed_surface}, \ref{teo:inv_cp} and Remark \ref{teo:ovaloide}.
\end{proof}

Given an ovaloid $S$, let $N_E : S \to \mathbb{S}^2$ be the Gauss map of $S$ viewed as a surface on Euclidean 3-space and take $N_E$ pointing to the interior of $S$. For $p \in S$ and $r > 0$, define $S_p(r)$ to be the Euclidean sphere with center $p + r N_E(p)$ and radius $r$. Thus, $S_p(r)$ and $S$ are tangent at $p$ for any $r > 0$.

\begin{lem}\label{prop:R}
	Let $S \subset \R^3_1$ be an ovaloid. Then there exist $R > 0$ such that $S \setminus \{p\}$ is contained inside $S_p(R)$, for all $p \in S$.
\end{lem}

\begin{proof}
	Let $p$ be a point in $S$ and $d$ the Euclidean distance in $\R^3$. Denote by $\theta (q) \in [0,\pi]$ the angle between the vectors $q-p$ and $N_E(p)$, for all $q \in S \setminus \{p\}$.
		
	Since the contact between $S$ and $T_pS$ in $p$ is $A_1^+$, there is a neighborhood $U$ of $p$ in $\R^3$ such that, in $U$, no point of the interior of $S$ belongs to $T_pS$ (considering $T_pS$ as the plane that passes through $p$). Suppose some $q \in S \setminus \{p\}$ belongs to $T_pS$. For sufficiently small $\epsilon>0$, the point $(1-\epsilon) p + \epsilon q$ belong to $U$ and $T_pS$, since the segment $[p,q]$ is contained in $ T_pS$. Therefore, $[p,q]$ is not contained within $S$, which contradicts the convexity of $S$. Therefore, $S \cap T_pS = \{p\}$ and $\theta (q) \in [0,\pi/2)$ for all $q \in S \setminus \{p\}$.
	
	Given $q \in S \setminus \{p\}$, let $q'$ be the orthogonal projection of $q$ to the line passing through $p$ with direction $N_E(p)$. We have
	\begin{equation}\label{eq:dist_r}
		\begin{array}{ccl}
			d(q,p+rN_E(p))^2 & = & d(q,q')^2+d(q',p+rN_E(p))^2\\
			& = & d(q,q')^2+\left(d(p,p+rN_E(p))-d(p,q')\right)^2\\
			& = & r^2-2 d(q) r \cos (\theta (q))+d(q)^2.
		\end{array}
	\end{equation}
	If $2r > d(q)^2/\cos (\theta (q))$, it follows from (\ref{eq:dist_r}) that $d(q,p+rN_E(p))< r$. Since $S$ is compact,
	$$d_{max} = \sup \{d(p,q) : q \in S\} < \infty \quad {\rm and} \quad c = \inf \{ \cos (\theta(q)) : q \in S \} > 0$$
	because $\theta (q) \in [0,\pi/2)$. Let $R(p) = d_{max}^2/2c$. For all $r>R(p)$ the set $S \setminus \{p\}$ is inside $S_p(r)$. Finally, let $R = \sup \{R(p) : p \in S\}+1$, which is finite because $S$ is compact. Therefore, $S \setminus \{p\}$ is inside $S_p(R)$ for all $p$.
\end{proof}

\begin{theo}\label{teo:inv_ov}
	Given an ovaloid $S \subset \R^3_1$, there is a translation $T: \R^3_1 \to \R^3_1$ such that $i_M (T(S))$ is an ovaloid.
\end{theo}

\begin{proof}
	Take $R>0$ given by the Lemma \ref{prop:R}. From the Corollary \ref{cor:mSac_ov} it is possible to obtain a translation $T: \R^3_1 \to \R^3_1$ such that the Möbius inversion in $\R_1^3$ of $T(S_p(R))$ is an ovaloid for all $p \in S$.
	
	It follows from item (3) of Theorem \ref{prop:inv_compac} that $T(S_p(R))$ does not intersect $LC$ for all $p \in S$ and also $T(S)$ do not intersect $LC$. Thus, $i_M (T(S))$ is a closed surface.
	
	Given $p \in S$, $i_M$ is a diffeomorphism between the interior of $T(S_p(R))$ and the interior of $i_M(T(S_p(R)))$ because $T(S_p(R)) \cap LC = \emptyset$. So $i_M (T(S \setminus \{p\}))$ is contained inside $i_M(T(S_p(R)))$ and $i_M(T(S))$ is tangent to $i_M (T(S_p(R)))$ in $i_M(T(p))$. Since $i_M(T(S_p(R))) $ is an ovaloid, the contact between $i_M(T(S_p(R)))$ and its tangent plane at $i_M(T(p))$ is a $A_1^+$, then the contact of $i_M(T(S))$ and its tangent plane at $i_M(T(p))$ is also $A_1^+$. Therefore, $i_M(T(S))$ is an ovaloid.
\end{proof}

To exemplify the Theorem \ref{teo:inv_ov}, consider again the Euclidean sphere $S(p_0,r)$ with $p_0 = (2,0,0)$ and $r=1$, then $i_M(S( p_0,r))$ is not an ovaloid as noted earlier (see Figure \ref{fig:inv_esfera}). However, taking the translation $T(x,y,z) = (x,y,z)+(2,0,0)$, then $i_M (T(S(p_0,r)))$ is an ovaloid (see Figure \ref {fig:convex}).

\begin{figure} [h!]
	\centering
	\includegraphics[scale=0.6]{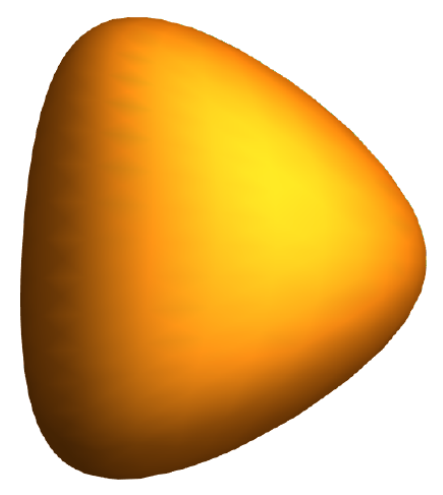}
	\caption{The surface $S_M(p_0,r)$ with $p_0=(2,0,0)$ and $r=1$.}
	\label{fig:convex}
\end{figure}

\begin{acknow}
This work was financed by the doctoral grant Coordenação de Aperfeiçoamento de Pessoal de Nível Superior – Brasil (CAPES) – Finance Code 001, and was supervised by Farid Tari.
\end{acknow}


\noindent
MACF: Instituto de Ci\^encias Matem\'aticas e de Computa\c{c}\~ao - USP, Avenida Trabalhador s\~ao-carlense, 400 - Centro, CEP: 13566-590 - S\~ao Carlos - SP, Brazil.\\
E-mail: marcoant\_couto@hotmail.com

\end{document}